\documentclass{article}
\usepackage{amssymb}
\usepackage{amsfonts}
\usepackage{amsmath}

\newtheorem{theorem}{Theorem}

\newtheorem{corollary}[theorem]{Corollary}

\newtheorem{definition}[theorem]{Definition}
\newtheorem{example}[theorem]{Example}

\newtheorem{lemma}[theorem]{Lemma}

\newtheorem{proposition}[theorem]{Proposition}
\newtheorem{remark}[theorem]{Remark}

\newenvironment{proof}[1][Proof]{\noindent\textbf{#1.} }{\ \rule{0.5em}{0.5em}}
 \font\tenmsb=msbm10 
scaled\magstep2

\input{tcilatex}
\begin{document}

\title{New results on $q$-positivity}
\author{Y.\ Garc\'{\i}a Ramos$^{1}$, J.E. Mart\'{\i}nez-Legaz$^{2}$\thanks{%
This author has been supported by the MICINN of Spain, grant
MTM2008-06695-C03-03, by the Barcelona Graduate School of Economics and by
the Government of Catalonia. He is affiliated to MOVE (Markets,
Organizations and Votes in Economics).}, S. Simons$^{3}$ \\
$^{1}$Centro de Investigaci\'{o}n de la Universidad del Pac\'{\i}fico, Lima,
Peru\\
$^{2}$Universitat Aut\`{o}noma de Barcelona, Spain\\
$^{3}$University of California, Santa Barbara, USA}
\date{}
\maketitle

\begin{abstract}
In this paper we discuss symmetrically self-dual spaces, which are simply
real vector spaces with a symmetric bilinear form. Certain subsets of the
space will be called $q$-positive, where $q$ is the quadratic form induced
by the original bilinear form. The notion of $q$-positivity generalizes the
classical notion of the monotonicity of a subset of a product of a Banach
space and its dual. Maximal $q$-positivity then generalizes maximal
monotonicity. We discuss concepts generalizing the representations of
monotone sets by convex functions, as well as the number of maximally $q$%
-positive extensions of a $q$-positive set. We also discuss symmetrically
self-dual Banach spaces, in which we add a Banach space structure, giving
new characterizations of maximal $q$-positivity. The paper finishes with two
new examples.
\end{abstract}

\section{Introduction}

In this paper we discuss symmetrically self-dual spaces, which are simply
real vector spaces with a symmetric bilinear form. Certain subsets of the
space will be called $q$-positive, where $q$ is the quadratic form induced
by the original bilinear form. The notion of $q$-positivity generalizes the
classical notion of the monotonicity of a subset of a product of a Banach
space and its dual. Maximal $q$-positivity then generalizes maximal
monotonicity.

A modern tool in the theory of monotone operators is the representation of
monotone sets by convex functions. We extend this tool to the setting of $q$%
-positive sets. We discuss the notion of the intrinsic conjugate of a proper
convex function on an SSD\ space. To each nonempty subset of an SSD space,
we associate a convex function, which generalizes the function originally
introduced by Fitzpatrick for the monotone case in \cite{Fi88}. In the same
paper he posed a problem on convex representations of monotone sets, to
which we give a partial solution in the more general context of this paper.

We prove that maximally $q$-positive convex sets are always affine, thus
extending a previous result in the theory of monotone operators \cite{BWY08,
MAS09}.

We discuss the number of maximally $q$-positive extensions of a $q$-positive
set. We show that either there are an infinite number of such extensions or
a unique extension, and in the case when this extension is unique we
characterize it. As a consequence of this characterization, we obtain a
sufficient condition for a monotone set to have a unique maximal monotone
extension to the bidual.

We then discuss symmetrically self-dual Banach spaces, in which we add a
Banach space structure to the bilinear structure already considered. In the
Banach space case, this corresponds to considering monotone subsets of the
product of a reflexive Banach space and its dual. We give new
characterizations of maximally $q$-positive sets, and of minimal convex
functions bounded below by $q.$

We give two examples of $q$-positivity: Lipschitz mappings between Hilbert
spaces, and closed sets in a Hilbert space.

\section{Preliminaries}

We will work in the setting of symmetrically self-dual spaces, a notion
introduced in \cite{Si08}. A \textit{symmetrically self-dual (SSD) space} is
a pair $(B,\lfloor \cdot ,\cdot \rfloor )$ consisting of a nonzero real
vector space $B$ and a symmetric bilinear form $\lfloor \cdot ,\cdot \rfloor
:B\times B\rightarrow \hbox{\tenmsb R}$. The bilinear form $\lfloor \cdot
,\cdot \rfloor $ induces the quadratic form $q$ on $B$ defined by $q(b)=%
\frac{1}{2}\lfloor b,b\rfloor $. A nonempty set $A\subset B$ is called $q$-%
\textit{positive} \cite[Definition 19.5]{Si08} if $b,c\in A\Rightarrow
q(b-c)\geq 0$. A set $M\subset B$ is called \textit{maximally }$q$-positive 
\cite[Definition 20.1]{Si08} if it is $q$-positive and not properly
contained in any other $q$-positive set. Equivalently, a $q$-positive set $A$
is maximally\textit{\ }$q$-positive if every $b\in B$ which is $q$-\textit{%
positively related} to A (i.e. $q(b-a)\geq 0$ for every $a\in A$) belongs to 
$A$. The set of all elements of $B$ that are $q$-positively related to $A$
will be denoted by $A^{\pi }$. The closure of $A$ with respect to the
(possibly non Hausdorff) weak topology $w\left( B,B\right) $ will be denoted
by $A^{w}.$

Given an arbitrary nonempty set $A\subset B$, the function $\Phi
_{A}:B\rightarrow \hbox{\tenmsb R}\cup \{+\infty \}$ is defined by 
\begin{equation*}
\Phi _{A}(x):=q(x)-\inf_{a\in A}q(x-a)=\sup_{a\in A}\{\lfloor x,a\rfloor
-q(a)\}.
\end{equation*}%
This generalizes the \textit{Fitzpatrick function} from the theory of
monotone operators. It is easy to see that $\Phi _{A}$ is a proper $w\left(
B,B\right) $-lsc convex function. If $M$ is maximally $q$-positive then 
\begin{equation}
\Phi _{M}(b)\geq q(b),\quad \forall \;b\in B,  \label{equation1}
\end{equation}%
and 
\begin{equation}
\Phi _{M}(b)=q(b)\Leftrightarrow b\in M.  \label{equation2}
\end{equation}%
A useful characterization of $A^{\pi }$ is the following: 
\begin{equation}
b\in A^{\pi }\text{ if and only if }\Phi _{A}(b)\leq q(b).  \label{mu}
\end{equation}%
The set of all proper convex functions $f:B\rightarrow \hbox{\tenmsb R}\cup
\{+\infty \}$ satisfying $f\geq q$ on $B$ will be denoted by $\mathcal{PC}%
_{q}(B)$ and, if $f\in \mathcal{PC}_{q}(B),$

\begin{equation}
\mathcal{P}_{q}(f):=\{{b\in B:f(b)=q(b)}\}{.}  \label{equation3}
\end{equation}

We will say that the convex function $f:B\rightarrow \hbox{\tenmsb R}\cup
\{+\infty \}$ is a $q$-\textit{representation} of a nonempty set $A\subset B$
if $f\in \mathcal{PC}_{q}(B)$ and $\mathcal{P}_{q}(f)=A$. In particular, if $%
A\subset B$ admits a $q$-representation, then it is $q$-positive \cite[Lemma
19.8]{Si08}. The converse is not true in general, see for example \cite[%
Remark 6.6]{Si08}.

A $q$-positive set in an SSD\ space having a $w\left( B,B\right) $-lsc $q$%
-representation will be called $q$-representable ($q$-representability is
identical with $\mathcal{S}$-$q$-positivity as defined in \cite[Def. 6.2]%
{Si07} in a more restrictive situation). By (\ref{equation1}) and (\ref%
{equation2}), every maximally $q$-positive set is $q$-representable.

If $B$ is a Banach space, we will denote by $\langle \cdot ,\cdot \rangle $
the duality products between $B$ and $B^{\ast }$ and between $B^{\ast }$ and
the bidual space $B^{\ast \ast }$, and the norm in $B^{\ast }$ will be
denoted by $\Vert \cdot \Vert $ as well. The topological closure, the
interior and the convex hull of a set $A\subset B$ will be denoted
respectively by $\overline{A}$, $intA$ and $convA$. The indicator function $%
\delta _{A}:B\rightarrow \hbox{\tenmsb R}\cup \{+\infty \}$ of $A\subset B$
is defined by%
\begin{equation*}
\delta _{A}\left( x\right) :=\left\{ 
\begin{array}{cc}
0 & \quad \text{if }x\in A \\ 
+\infty & \quad \text{if }x\notin A%
\end{array}%
.\right.
\end{equation*}%
The convex envelope of $f:B\rightarrow \hbox{\tenmsb R}\cup \{+\infty \}$
will be denoted by $conv~f$.

\section{SSD spaces}

Following the notation of \cite{JE08}, for a proper convex function $%
f:B\rightarrow \hbox{\tenmsb R}\cup \{+\infty \}$, we will consider its
intrinsic (Fenchel) conjugate $f^{@}:B\rightarrow \hbox{\tenmsb R}\cup
\{+\infty \}$ with respect to the pairing $\lfloor \cdot ,\cdot \rfloor :$%
\begin{equation*}
f^{@}(b):=\sup \{\lfloor c,b\rfloor -f(c):c\in B\}.
\end{equation*}

\begin{proposition}[\protect\cite{Si08,JE08}]
\label{prop0}Let $A$ be a $q$-positive subset of an SSD\ space $B$. The
following statements hold:

(1) For every $b \in B$, $\Phi_A(b) \leq \Phi^@_A(b)$ and $q(b) \leq
\Phi^@_A(b)$;

(2) For every $a \in A$, $\Phi_A(a) = q(a) = \Phi^@_A(a)$;


(3) $\Phi _{A}^{@}$ is the largest $w\left( B,B\right) $-lsc convex function
majorized by $q$ on $A$;

(4) $A$ is $q$-representable if, and only if, $\mathcal{P}_{q}(\Phi
_{A}^{@})\subset A;$

(5) $A$ is $q$-representable if, and only if, for all $b\in B$ such that,
for all $c\in B,$ $\lfloor c,b\rfloor \leq \Phi _{A}(c)+q\left( b\right) ,$
one has $b\in A.$
\end{proposition}

\begin{proof}
\textit{(1)} and \textit{(2)}. Let $a\in A$ and $b\in B.$ Since $A$ is $q$%
-positive, the infimum $\inf_{a^{\prime }\in A}q(a-a^{\prime })$ is attained
at $a^{\prime }=a;\ $hence we have the first equality in \textit{(2)}. Using
this equality, one gets%
\begin{equation*}
\Phi _{A}^{@}(b)=\sup_{c\in B}\{\lfloor c,b\rfloor -\Phi _{A}(c)\}\geq
\sup_{a\in A}\{\lfloor a,b\rfloor -\Phi _{A}(a)\}=\sup_{a\in A}\{\lfloor
a,b\rfloor -q(a)\}=\Phi _{A}(b),
\end{equation*}%
which proves the first inequality in \textit{(1)}. In view of this
inequality, given that $\Phi _{A}^{@}(b)=\sup_{c\in B}\{\lfloor c,b\rfloor
-\Phi _{A}(c)\}\geq \lfloor b,b\rfloor -\Phi _{A}(b)=2q\left( b\right) -\Phi
_{A}(b),$ we have $\Phi _{A}^{@}(b)\geq \max \left\{ 2q\left( b\right) -\Phi
_{A}(b),\Phi _{A}(b)\right\} =q\left( b\right) +\left\vert q\left( b\right)
-\Phi _{A}(b)\right\vert \geq q\left( b\right) ,$ so that the second
inequality in \textit{(1)} holds true. From the definition of $\Phi _{A}$ it
follows that $\Phi _{A}(c)\geq \lfloor c,a\rfloor -q(a)$ for every $c\in B;$
therefore%
\begin{equation*}
\Phi _{A}^{@}(a)=\sup_{c\in B}\{\lfloor c,a\rfloor -\Phi _{A}(c)\}\leq
q\left( a\right) .
\end{equation*}%
From this inequality and the second one in \textit{(1)} we obtain the second
equality in \textit{(2)}.

\textit{(3)}. Let $f$ be a $w\left( B,B\right) $-lsc convex function
majorized by $q$ on $A$. Then, for all $b\in B,$%
\begin{eqnarray*}
\Phi _{A}(b) &=&\sup_{a\in A}\left\{ \left\lfloor b,a\right\rfloor -q\left(
a\right) \right\} =\sup_{a\in A}\left\{ \left\lfloor a,b\right\rfloor
-q\left( a\right) \right\} \\
&\leq &\sup_{a\in A}\left\{ \left\lfloor a,b\right\rfloor -f\left( a\right)
\right\} \leq \sup_{c\in B}\left\{ \left\lfloor c,b\right\rfloor -f\left(
c\right) \right\} =f^{@}\left( b\right) .
\end{eqnarray*}%
Thus $\Phi _{A}\leq f^{@}$ on $B.$ Consequently $f^{@@}\leq $ $\Phi _{A}^{@}$
on $B.$ Since $f$ is $w\left( B,B\right) $-lsc, from the (non Hausdorff)
Fenchel-Moreau theorem \cite[Theorem 10.1]{Si11}, $f\leq $ $\Phi _{A}^{@}$
on $B.$

\textit{(4)}. We note from \textit{(1)} and \textit{(2)} that $\Phi
_{A}^{@}\in \mathcal{PC}_{q}(B)$ and $A\subset \mathcal{P}_{q}(\Phi
_{A}^{@}).$ It is clear from these observations that if $\mathcal{P}%
_{q}(\Phi _{A}^{@})\subset A$ then $\Phi _{A}^{@}$ is a $w\left( B,B\right) $%
-lsc $q$-representation of $A.$ Suppose, conversely, that $A$ is $q$%
-representable, so that there exists a $w(B;B)$-lsc function $f\in \mathcal{%
PC}_{q}(B)$ such that $\mathcal{P}_{q}(f)=A.$ It now follows from \textit{(3)%
} that $f\leq \Phi _{A}^{@}$ on A, and so $\mathcal{P}_{q}(\Phi
_{A}^{@})\subset \mathcal{P}_{q}(f)=A.$

(\textit{5)}. This statement follows from \textit{(4)}, since the inequality 
$\lfloor c,b\rfloor \leq \Phi _{A}(c)+q\left( b\right) $ holds for all $c\in
B$ if, and only if, $b\in \mathcal{P}_{q}(\Phi _{A}^{@}).$
\end{proof}

\bigskip

The next results should be compared with \cite[Theorems 6.3.(b) and 6.5.(a)]%
{Si07}.

\begin{corollary}
\label{q-repr hull}Let $A$ be a $q$-positive subset of an SSD\ space $B$.
Then $\mathcal{P}_{q}(\Phi _{A}^{@})$ is the smallest $q$-representable
superset of $A.$
\end{corollary}

\begin{proof}
By Proposition \ref{prop0}\textit{.(2)}, $\mathcal{P}_{q}(\Phi _{A}^{@})$ is
a $q$-representable superset of $A.$ Let $C$ be a $q$-representable superset
of $A.$ Since $A\subset C,$ we have $\Phi _{A}\leq \Phi _{C}$ and hence $%
\Phi _{C}^{@}\leq \Phi _{A}^{@}.$ Therefore, by Proposition \ref{prop0}%
\textit{.(4), }$\mathcal{P}_{q}(\Phi _{A}^{@})\subset \mathcal{P}_{q}(\Phi
_{C}^{@})\subset C.$
\end{proof}

\begin{corollary}
Let $A$ be a $q$-positive subset of an SSD\ space $B,$ and denote by $C$ the
smallest $q$-representable superset of $A.$ Then $\Phi _{C}=\Phi _{A}.$
\end{corollary}

\begin{proof}
Since $A\subset C,$ we have $\Phi _{A}\leq \Phi _{C}.$ On the other hand, by
Corollary \ref{q-repr hull}, $C=\mathcal{P}_{q}(\Phi _{A}^{@});$ hence $\Phi
_{A}^{@}$ is majorized by $q$ on $C$. Therefore, by Proposition \ref{prop0}%
\textit{.(3)}, $\Phi _{A}^{@}\leq \Phi _{C}^{@}.$ Since $\Phi _{A}\ $and $%
\Phi _{C}$ are $w\left( B,B\right) $-lsc, from the (non Hausdorff)
Fenchel-Moreau theorem \cite[Theorem 10.1]{Si11}, $\Phi _{C}=\Phi
_{C}^{@@}\leq \Phi _{A}^{@@}=\Phi _{A}.$ We thus have $\Phi _{C}=\Phi _{A}.$
\end{proof}

\bigskip

We continue with a result about the domain of $\Phi _{A}^{@}$ which will be
necessary in the sequel.

\begin{lemma}[about the domain of $\Phi _{A}^{@}$]
\label{laboutd} Let $A$ be a $q$-positive subset of an SSD\ space $B$. Then, 
\begin{equation*}
convA\subset \hbox{\rm dom}\Phi _{A}^{@}\subset conv^{w}A.
\end{equation*}
\end{lemma}

\begin{proof}
Since $\Phi _{A}^{@}$ coincides with $q$ in $A$, we have that $A\subset %
\hbox{\rm dom}\Phi _{A}^{@}$, hence from the convexity of $\Phi _{A}^{@}$ it
follows that 
\begin{equation*}
convA\subset \hbox{\rm dom}\Phi _{A}^{@}.
\end{equation*}%
On the other hand, from Proposition \ref{prop0}\textit{(3)} $\Phi
_{A}^{@}+\delta _{conv^{w}A}\leq \Phi _{A}^{@}$, because\newline
$\Phi _{A}^{@}+\delta _{conv^{w}A}$ is $w\left( B,B\right) $-lsc, convex and
majorized by $q$ on $A$. Thus, 
\begin{equation*}
\hbox{\rm dom}\Phi _{A}^{@}\subset \hbox{\rm dom}\left( \Phi _{A}^{@}+\delta
_{conv^{w}A}\right) \subset conv^{w}A.
\end{equation*}%
This finishes the proof.
\end{proof}

\subsection{On a problem posed by Fitzpatrick}

Let $B$ be an SSD$\ $space and $f:B\rightarrow \hbox{\tenmsb R}\cup
\{+\infty \}$ be a proper convex function. The generalized Fenchel-Young
inequality establishes that 
\begin{equation}
f(a)+f^{@}(b)\geq \lfloor a,b\rfloor ,\quad \forall \,a,b\in B.  \label{FY}
\end{equation}%
We define the $q$-\textit{subdifferential} of $f$ at $a\in B$ by%
\begin{equation*}
\partial _{q}f(a):=\left\{ b\in B:f(a)+f^{@}(b)=\lfloor a,b\rfloor \right\}
\end{equation*}%
and the set%
\begin{equation*}
G_{f}:=\{b\in B:b\in \partial _{q}f(b)\}.
\end{equation*}%
In this Subsection we are interested in identifying sets $A\subset B$ with
the property that $G_{\Phi _{A}}=A.$ The problem of characterizing such sets
is an abstract version of an open problem on monotone operators posed by
Fitzpatrick \cite[Problem 5.2]{Fi88}.

\begin{proposition}
\label{prop1}Let $B$ be an SSD$\ $space and $f:B\rightarrow \hbox{\tenmsb R}%
\cup \{+\infty \}$ be a $w\left( B,B\right) $-lsc proper convex function
such that $G_{f}\neq \emptyset $. Then the set $G_{f}$ is $q$-representable.
\end{proposition}

\begin{proof}
Taking the $w\left( B,B\right) $-lsc proper convex function $h:=\frac{1}{2}%
(f+f^{@})$, we have that 
\begin{equation*}
G_{f}=\mathcal{P}_{q}(h).
\end{equation*}
\end{proof}

\begin{theorem}
\label{theorem1}Let $A$ be a $q$-positive subset of an SSD space $B$. Then

(1) $A\subset \mathcal{P}_{q}(\Phi _{A}^{@})\subset G_{\Phi _{A}}\subset
A^{\pi }\cap conv^{w}A$;

(2) If $A$ is convex and $w\left( B,B\right) $-closed,%
\begin{equation*}
A=G_{\Phi _{A}};
\end{equation*}

(3) If $A$ is maximally $q$-positive,%
\begin{equation*}
A=G_{\Phi _{A}}.
\end{equation*}
\end{theorem}

\begin{proof}
\textit{(1)}. By Proposition \ref{prop0}\textit{(2)}, we have the first
inclusion in \textit{(1)}. Let $b\in \mathcal{P}_{q}(\Phi _{A}^{@})$. Since $%
\Phi _{A}(b)\leq \Phi _{A}^{@}(b)=q(b)$, we get 
\begin{equation*}
2q(b)\leq \Phi _{A}(b)+\Phi _{A}^{@}(b)\leq 2q(b).
\end{equation*}%
It follows that $b\in G_{\Phi _{A}}$. This shows that $\mathcal{P}_{q}(\Phi
_{A}^{@})\subset G_{\Phi _{A}}$.Using Proposition \ref{prop0}\textit{(1)},
we infer that for any $a\in G_{\Phi _{A}}$, $\Phi _{A}(a)\leq q(a)$, so $%
G_{\Phi _{A}}\subset A^{\pi }$. On the other hand, since $G_{\Phi
_{A}}\subset \hbox{\rm dom}\Phi _{A}^{@}$, Lemma \ref{laboutd} implies that $%
G_{\Phi _{A}}\subset conv^{w}A$. This proves the last inclusion in \textit{%
(1)}.

\textit{(2)}.\textit{\ }This is immediate from \textit{(1)} since $%
conv^{w}A=A.$

\textit{(3)}. This follows directly from Proposition \ref{prop1} and \textit{%
(1)}.
\end{proof}

\begin{proposition}
\label{prop2} Let $A$ be a nonempty subset of an SSD$\ $space $B$ and let $D$
be a $w\left( B,B\right) $-closed convex subset of $B$ such that%
\begin{equation}
\Phi _{A}(b)\geq q(b)\text{\qquad }\forall \,b\in D.  \label{important}
\end{equation}%
Suppose that $A^{\pi }\cap D\neq \emptyset .$ Then $A^{\pi }\cap D$ is $q$%
-representable.
\end{proposition}

\begin{proof}
We take $f=\Phi _{A}+\delta _{D}$; this function is $w\left( B,B\right) $%
-lsc, proper (because $A^{\pi }\neq \emptyset $) and convex. Let $b\in B$ be
such that $f(b)\leq q(b)$, so 
\begin{equation*}
\Phi _{A}(b)\leq q(b)\text{ and }b\in D.
\end{equation*}%
This implies that $b\in A^{\pi }\cap D$. From (\ref{important}) we infer
that $f(b)=\Phi _{A}(b)=q(b)$. It follows that $f\in \mathcal{PC}_{q}(B)$.
It is easy to see that $f$ is a $q$-representative function for $A^{\pi
}\cap D$.
\end{proof}

\begin{proposition}
Let $A$ be a $q$-positive subset of an SSD space $B$. If $C=A^{\pi }\cap
conv^{w}A$ is $q$-positive, then%
\begin{equation*}
C=G_{\Phi _{C}}=C^{\pi }\cap conv^{w}C.
\end{equation*}
\end{proposition}

\begin{proof}
Clearly $conv^{w}A\supset C$, from which $conv^{w}A\supset conv^{w}C$. Since 
$C\supset A,$ $A^{\pi }\supset C^{\pi }$. Thus $C=A^{\pi }\cap
conv^{w}A\supset C^{\pi }\cap conv^{w}C$. However, from Theorem \ref%
{theorem1}\textit{(1)}, $C\subset G_{\Phi _{C}}\subset C^{\pi }\cap
conv^{w}C $.
\end{proof}

\begin{proposition}
\label{prop5}Let $A$ be a $q$-positive subset of an SSD space $B$. If 
\begin{equation}
\Phi _{A}(b)\geq q(b)\text{\qquad }\forall \,b\in conv^{w}A,  \label{ineq}
\end{equation}%
then%
\begin{equation*}
G_{\Phi _{A}}=\mathcal{P}_{q}(\Phi _{A}^{@}).
\end{equation*}
\end{proposition}

\begin{proof}
It is clear from Theorem \ref{theorem1}\textit{(1)} and (\ref{ineq}) that,
for all $b\in G_{\Phi _{A}},$ $\Phi _{A}(b)=q(b)$; thus $\Phi
_{A}^{@}(b)=\left\lfloor b,b\right\rfloor -\Phi _{A}(b)=q(b)$, so $G_{\Phi
_{A}}\subset \mathcal{P}_{q}(\Phi _{A}^{@}).$ The opposite inclusion also
holds, according to Theorem \ref{theorem1}\textit{(1).}
\end{proof}

\begin{corollary}
Let $A$ be a $q$-positive subset of an SSD space $B$. If $\Phi _{A}\in 
\mathcal{PC}_{q}(B)$, then%
\begin{equation*}
G_{\Phi _{A}}=\mathcal{P}_{q}(\Phi _{A}^{@}).
\end{equation*}
\end{corollary}

\begin{proposition}
\label{prop6} Let $A$ be a $q$-representable subset of an SSD space $B$. If $%
\Phi _{A}(b)\geq q(b)$ for all $b\in conv^{w}A$, then 
\begin{equation*}
A=G_{\Phi _{A}}.
\end{equation*}
\end{proposition}

\begin{proof}
Since $A$ is a $q$-representable set, $A=\mathcal{P}_{q}(f)$ for some $%
w\left( B,B\right) $-lsc $f\in \mathcal{PC}_{q}(B).$ By Proposition \ref%
{prop0}\textit{(3)}, $f\leq \Phi _{A}^{@};$ hence, by Corollary \ref{incl}, $%
\mathcal{P}_{q}(f)\supset \mathcal{P}_{q}(\Phi _{A}^{@})\supset A=\mathcal{P}%
_{q}(f)$, so that $A=\mathcal{P}_{q}(\Phi _{A}^{@})$. The result follows by
applying Proposition \ref{prop5}.
\end{proof}

\begin{lemma}
\label{lemma2}Let $A$ be a $q$-positive subset of an SSD space $B$. If for
some topological vector space $Y$ there exists a $w\left( B,B\right) $%
-continuous linear mapping $f:B\rightarrow Y$ satisfying

(1) $f(A)$ is convex and closed,

(2) $f(x)=0$ implies $q(x)=0$,

\noindent then%
\begin{equation}
\Phi _{A}(b)\geq q(b)\text{\qquad }\forall \,b\in conv^{w}A.
\label{ineq_on_hull}
\end{equation}
\end{lemma}

\begin{proof}
Since 
\begin{equation*}
f(A)\subset f\left( conv^{w}A\right) \subset \overline{conv}f(A)=f(A),
\end{equation*}%
it follows that 
\begin{equation*}
f\left( conv^{w}A\right) =f(A).
\end{equation*}%
Let $b\in conv^{w}A$. Then there exists $a\in A$ such that $f(b)=f(a)$,
hence $f(a-b)=0$. By \textit{2,} $q(a-b)=0$, and so we obviously have (\ref%
{ineq_on_hull}).
\end{proof}

\begin{corollary}
Let $T:X\rightrightarrows X^{\ast }$ be a representable monotone operator on
a Banach space $X$. If $DomT\,(RanT)$ is convex and closed, then 
\begin{equation*}
T=G_{\varphi _{T}}.
\end{equation*}
\end{corollary}

\begin{proof}
Take $f=P_{X}$ or $f=P_{X^{\ast }},$ the projections onto $X$ and $X^{\ast
}, $ respectively, in Lemma \ref{lemma2} and apply Proposition \ref{prop6}.
Notice that when $X\times X^{\ast }$ is endowed with the topology $w(X\times
X^{\ast },X^{\ast }\times X)$, $P_{X}$ and $P_{X^{\ast }}$ are continuous
onto $X$ with its weak topology and $X^{\ast }$ with the weak$^{\ast }$
topology, respectively.
\end{proof}

\subsection{Maximally $q$-positive convex sets}

The following result extends \cite[Lemma 1.5]{MAS09} (see also \cite[Thm. 4.1%
]{BWY08}).

\begin{theorem}
\label{15}Let $A$ be a maximally $q$-positive convex set in an SSD space $B.$
Then $A$ is actually affine.
\end{theorem}

\begin{proof}
Take $x_{0}\in A.$ Clearly, the set $A-x_{0}$ is also maximally $q$-positive
and convex. To prove that $A$ is affine, we will prove that $A-x_{0}$ is a
cone, that is,%
\begin{equation}
\lambda \left( x-x_{0}\right) \in A-x_{0}\text{\qquad for all }x\in A\text{
and }\lambda \geq 0,  \label{cone}
\end{equation}%
and that it is symmetric with respect to the origin, that is,%
\begin{equation}
-\left( x-x_{0}\right) \in A-x_{0}\text{\qquad for all }x\in A\text{.}
\label{symmetric}
\end{equation}%
Let $x\in A$ and $\lambda \geq 0.$ If $\lambda \leq 1,$ then $\lambda \left(
x-x_{0}\right) =\lambda x+\left( 1-\lambda \right) x_{0}-x_{0}\in A-x_{0},$
since $A$ is convex. If $\lambda \geq 1,$ for every $y\in A$ we have $%
q\left( \lambda \left( x-x_{0}\right) -\left( y-x_{0}\right) \right)
=\lambda ^{2}q\left( x-\left( \frac{1}{\lambda }\left( y-x_{0}\right)
+x_{0}\right) \right) \geq 0,$ since $\frac{1}{\lambda }\left(
y-x_{0}\right) \in A-x_{0}.$ Hence, as $A-x_{0}$ is maximally $q$-positive, $%
\lambda \left( x-x_{0}\right) \in A-x_{0}$ also in this case. This proves (%
\ref{cone}). To prove (\ref{symmetric}), let $x,y\in A.$ Then $q\left(
-\left( x-x_{0}\right) -\left( y-x_{0}\right) \right) =q\left( \left(
x+y-x_{0}\right) -x_{0}\right) \geq 0,$ since $x+y-x_{0}\in A$ (as $A-x_{0}$
is a convex cone) and $x_{0}\in A$. Using that $A-x_{0}$ is maximally $q$%
-positive, we conclude that $-\left( x-x_{0}\right) \in A-x_{0},$ which
proves (\ref{symmetric}).
\end{proof}

\subsection{About the number of maximally $q$-positive extensions of a $q$%
-positive set}

\begin{proposition}
\label{starshaped}Let $x_{1},x_{2}\in B$ be such that%
\begin{equation}
q(x_{1}-x_{2})\leq 0.  \label{three}
\end{equation}%
Then $\lambda x_{1}+\left( 1-\lambda \right) x_{2}\in \left\{
x_{1},x_{2}\right\} ^{\pi \pi }$ for every $\lambda \in \lbrack 0,1]$.
\end{proposition}

\begin{proof}
Let $x\in \left\{ x_{1},x_{2}\right\} ^{\pi }.$ Since%
\begin{equation*}
q(x_{1}-x_{2})=q\left( (x_{1}-x)-(x_{2}-x)\right) =q(x_{1}-x)-\left\lfloor
x_{1}-x,x_{2}-x\right\rfloor +q(x_{2}-x),
\end{equation*}%
(\ref{three}) implies that 
\begin{equation*}
\left\lfloor x_{1}-x,x_{2}-x\right\rfloor \geq q(x_{1}-x)+q(x_{2}-x).
\end{equation*}%
Then, writing $x_{\lambda }:=\lambda x_{1}+\left( 1-\lambda \right) x_{2}$,%
\begin{eqnarray*}
q\left( x_{\lambda }-x\right) &=&q\left( \lambda \left( x_{1}-x\right)
+\left( 1-\lambda \right) \left( x_{2}-x\right) \right) \\
&=&\lambda ^{2}q(x_{1}-x)+\lambda (1-\lambda )\left\lfloor
x_{1}-x,x_{2}-x\right\rfloor +(1-\lambda )^{2}q(x_{2}-x) \\
&\geq &\lambda ^{2}q(x_{1}-x)+\lambda (1-\lambda )\left(
q(x_{1}-x)+q(x_{2}-x)\right) \\
&&+(1-\lambda )^{2}q(x_{2}-x) \\
&=&\lambda q(x_{1}-x)+(1-\lambda )q(x_{2}-x)\geq 0.
\end{eqnarray*}
\end{proof}

\bigskip

We will use the following lemma:

\begin{lemma}
\label{third polar}Let $A\subset B.$ Then $A^{\pi \pi \pi }=A^{\pi }.$
\end{lemma}

\begin{proof}
Since $q$ is an even function, from the definition of $A^{\pi }$ it follows
that $A\subset A^{\pi \pi }.$ Replacing $A$ by $A^{\pi }$ in this inclusion,
we get $A^{\pi }\subset A^{\pi \pi \pi }.$ On the other hand, since the
mapping $A\longmapsto A^{\pi }$ is inclusion reversing, from $A\subset
A^{\pi \pi }$ we also obtain $A^{\pi \pi \pi }\subset A^{\pi }.$
\end{proof}

\begin{proposition}
Let $A$ be a $q$-positive set. If $A$ has more than one maximally $q$%
-positive extension, then it has a continuum of such extensions.
\end{proposition}

\begin{proof}
Let $M_{1}$, $M_{2}$ be two different maximally $q$-positive extensions of $%
A $. By the maximality of $M_{1}$ and $M_{2}$, there exists $x_{1}\in M_{1}$
and $x_{2}\in M_{2}$ such that $q(x_{1}-x_{2})<0.$ Notice that $\left\{
x_{1},x_{2}\right\} \subset A^{\pi };$ hence, using proposition \ref%
{starshaped} and Lemma \ref{third polar}, we deduce that, for every $\lambda
\in \lbrack 0,1],$ $\lambda x_{1}+\left( 1-\lambda \right) x_{2}\in \left\{
x_{1},x_{2}\right\} ^{\pi \pi }\subset A^{\pi \pi \pi }=A^{\pi }.$ This
shows that, for each $\lambda \in \lbrack 0,1]$, $A\cup \{x_{\lambda }\},$
with $x_{\lambda }:=\lambda x_{1}+\left( 1-\lambda \right) x_{2}$, is a $q$%
-positive extension of $A$; since $q(x_{\lambda _{1}}-x_{\lambda
_{2}})=q\left( \left( \lambda _{1}-\lambda _{2}\right) \left(
x_{1}-x_{2}\right) \right) =\left( \lambda _{1}-\lambda _{2}\right)
^{2}q\left( x_{1}-x_{2}\right) <0$ for all $\lambda _{1},\lambda _{2}\in
\lbrack 0,1]$ with $\lambda _{1}\neq \lambda _{2}$, the result follows using
Zorn's Lemma.
\end{proof}

\subsection{Premaximally $q$-positive sets}

Let $\left( B,\lfloor \cdot ,\cdot \rfloor \right) $ be an SSD space.

\begin{definition}
\label{101}Let $P$ be a $q$--positive subset of $B$. We say that $P$ is 
\textsl{premaximally $q$--positive} if there exists a unique maximally $q$%
--positive superset of $P$. It follows from \cite[Lemma 5.4]{Si07} that this
superset is $P^{\pi }$ (which is identical with $P^{\pi \pi }$). The same
reference also implies that \textsl{%
\begin{equation}
P\ \hbox{is premaximally}\ q\hbox{--positive}\iff P^{\pi }\ \hbox{is}\ q%
\hbox{--positive}.  \label{201}
\end{equation}%
}
\end{definition}

\begin{lemma}
\label{102}Let $P$ be a $q$--positive subset of $B$ and 
\begin{equation}
\Phi _{P}\geq q\ \hbox{on}\ B.  \label{202}
\end{equation}%
Then $P$ is premaximally $q$--positive and $P^{\pi }=\mathcal{P}_{q}(\Phi
_{P})$.
\end{lemma}

\begin{proof}
Suppose that $M$ is a maximally $q$--positive subset of $B$ and $M\supset P$%
. Let $b\in M$. Since $M$ is $q$--positive, $b\in M^{\pi }\subset P^{\pi }$,
thus $\Phi _{P}(b)\leq q(b)$. Combining this with (\ref{202}), $\Phi
_{P}(b)=q(b)$, and so $b\in \mathcal{P}_{q}(\Phi _{P})$. Thus we have proved
that $M\subset \mathcal{P}_{q}(\Phi _{P})$. It now follows from the
maximality of $M$ and the $q$--positivity of $\mathcal{P}_{q}(\Phi _{P})$
that $P^{\pi }=\mathcal{P}_{q}(\Phi _{P})$.
\end{proof}

\bigskip

The next result contains a partial converse to Lemma \ref{102}.

\begin{lemma}
\label{103}Let $P$ be a premaximally $q$--positive subset of $B$. Then
either (\ref{202}) is true, or $P^{\pi }=\hbox{\rm dom}\,\Phi _{P}$ and $%
P^{\pi }$ is an affine subset of $B$.
\end{lemma}

\begin{proof}
Suppose that (\ref{202}) is not true. We first show that 
\begin{equation}
\hbox{\rm dom}\,\Phi _{P}\ \hbox{is}\ q\hbox{--positive}.  \label{203}
\end{equation}%
Since (\ref{202}) fails, we can first fix $b_{0}\in B$ such that $(\Phi
_{P}-q))(b_{0})<0$. Now let $b_{1},b_{2}\in \hbox{\rm dom}\,\Phi _{P}$. Let $%
\lambda \in \,]0,1[\,$. Then 
\begin{equation}
(\Phi _{P}-q)\left( (1-\lambda )b_{0}+\lambda b_{1}\right) \leq (1-\lambda
)\Phi _{P}(b_{0})+\lambda \Phi _{P}(b_{1})-q\left( (1-\lambda )b_{0}+\lambda
b_{1}\right) .  \label{204}
\end{equation}%
Since $\Phi _{P}(b_{1})\in \hbox{\tenmsb R}$ and quadratic forms on
finite--dimensional spaces are continuous, the right--hand expression in (%
\ref{204}) converges to $\Phi _{P}(b_{0})-q(b_{0})$ as $\lambda \rightarrow
0+$. Now $\Phi _{P}(b_{0})-q(b_{0})<0$ and so, for all sufficiently small $%
\lambda \in \,]0,1[\,$, $(\Phi _{P}-q)\left( (1-\lambda )b_{0}+\lambda
b_{1}\right) <0$, from which $(1-\lambda )b_{0}+\lambda b_{1}\in P^{\pi }$.
Similarly, for all sufficiently small $\lambda \in \,]0,1[\,$, $(1-\lambda
)b_{0}+\lambda b_{2}\in P^{\pi }$. Thus we can choose $\lambda _{0}\in
\,]0,1[\,$ such that both $(1-\lambda _{0})b_{0}+\lambda _{0}b_{1}\in P^{\pi
}$ and $(1-\lambda _{0})b_{0}+\lambda b_{2}\in P^{\pi }$. Since $P^{\pi }$
is $q$--positive, 
\begin{equation*}
0\leq q\left( [(1-\lambda _{0})b_{0}+\lambda _{0}b_{1}]-[(1-\lambda
_{0})b_{0}+\lambda b_{2}]\right) =\lambda _{0}^{2}q(b_{1}-b_{2}).
\end{equation*}%
So we have proved that, for all $b_{1},b_{2}\in \hbox{\rm dom}\,\Phi _{P}$, $%
q(b_{1}-b_{2})\geq 0$. This establishes (\ref{203}). Therefore, since $%
\hbox{\rm dom}\,\Phi _{P}\supset P$, we have $\hbox{\rm dom}\,\Phi
_{P}\subset P^{\pi }$. On the other hand, if $b\in P^{\pi }$, then $\Phi
_{P}(b)\leq q(p)$, and so $b\in \hbox{\rm dom}\,\Phi _{P}$. This completes
the proof that $P^{\pi }=\hbox{\rm dom}\,\Phi _{P}$. Finally, since $P^{\pi
}(=\hbox{\rm dom}\,\Phi _{P})$ is convex, Theorem \ref{15} implies that $%
P^{\pi }$ is an affine subset of $B$.
\end{proof}

\bigskip

Our next result is a new characterization of premaximally $q$--positive sets.

\begin{theorem}
\label{104}Let $P$ be a $q$--positive subset of $B$. Then $P$ is
premaximally $q$--positive if, and only if, either (\ref{202}) is true or $%
P^{\pi }$ is an affine subset of $B$.
\end{theorem}

\begin{proof}
\textquotedblleft Only if\textquotedblright\ is clear from Lemma \ref{103}.
If, on the other hand, (\ref{202}) is true then Lemma \ref{102}\ implies
that $P$ is premaximally $q$--positive. It remains to prove that if $P^{\pi
} $ is an affine subset of $B$ then $P$ is premaximally $q$--positive. So
let $P^{\pi }$ be an affine subset of $B$. Suppose that $b_{1},b_{2}\in
P^{\pi }$, and let $p\in P$. Since $P$ is $q$--positive, $p\in P^{\pi }$,
and since $P^{\pi }$ is affine, $p+b_{1}-b_{2}\in P^{\pi }$, from which $%
q(b_{1}-b_{2})=q([p+b_{1}-b_{2}]-p)\geq 0$. Thus we have proved that $P^{\pi
}$ is $q$--positive. It now follows from (\ref{201}) that $P$ is
premaximally $q$--positive.
\end{proof}

\begin{corollary}
Let $P$ be an affine $q$--positive subset of $B$. Then $P$ is premaximally $%
q $--positive if and only if $P^{\pi }$ is an affine subset of $B$.
\end{corollary}

\begin{proof}
In view of Theorem \ref{104}, we only need to prove the "only if" statement.
Assume that $P$ is premaximally $q$--positive. Since the family of affine
sets $A$ such that $P\subset A\subset P^{\pi }$ is inductive, by Zorn's
Lemma it has a maximal element $M.$ Let $b\in P^{\pi },$ $m_{1},m_{2}\in M,$ 
$p\in P$ and $\lambda ,\mu ,\nu \in \mathbb{R}$ be such that $\lambda +\mu
+\nu =1.$ If $\lambda \neq 0$ then $q\left( \lambda b+\mu m_{1}+\nu
m_{2}-p\right) =\lambda ^{2}q\left( b-\frac{1}{\lambda }\left( p-\mu
m_{1}-\nu m_{2}\right) \right) \geq 0,$ since $\frac{1}{\lambda }\left(
p-\mu m_{1}-\nu m_{2}\right) \in M\subset P^{\pi }$ and $P^{\pi }$ is $q$%
--positive (by \cite[Lemma 5.4]{Si07}). If, on the contrary, $\lambda =0$
then $q\left( \lambda b+\mu m_{1}+\nu m_{2}-p\right) =q\left( \mu m_{1}+\nu
m_{2}-p\right) \geq 0,$ because in this case $\mu m_{1}+\nu m_{2}\in
M\subset P^{\pi }.$ Therefore $\lambda b+\mu m_{1}+\nu m_{2}\in P^{\pi }.$
We have thus proved that the affine set generated by $M\cup \left\{
b\right\} $ is contained in $P^{\pi }.$ Hence, by the maximality of $M,$ we
have $b\in M,$ and we conclude that $P^{\pi }=M.$
\end{proof}

\begin{definition}
\label{105}Let $E$ be a nonzero Banach space and $A$ be a nonempty monotone
subset of $E\times E^{\ast }$. We say that $A$ is \textsl{of type (NI)} if, 
\begin{equation*}
\hbox{for all}\ (y^{\ast },y^{\ast \ast })\in E^{\ast }\times E^{\ast \ast
},\quad \inf\nolimits_{(a,a^{\ast })\in A}\langle a^{\ast }-y^{\ast },%
\widehat{a}-y^{\ast \ast }\rangle \leq 0.
\end{equation*}%
We define $\iota \colon \ E\times E^{\ast }\rightarrow E^{\ast }\times
E^{\ast \ast }$ by $\iota (x,x^{\ast })=(x^{\ast },\widehat{x})$, where $%
\widehat{x}$ is the canonical image of $x$ in $E^{\ast \ast }$. We say that $%
A$ is \textsl{unique} if there exists a unique maximally monotone subset $M$
of $E^{\ast }\times E^{\ast \ast }$ such that $M\supset \iota (A)$. We now
write $B:=E^{\ast }\times E^{\ast \ast }$ and define $\lfloor \cdot ,\cdot
\rfloor \colon \ B\times B\rightarrow \hbox{\tenmsb R}$ by $\left\lfloor
(x^{\ast },x^{\ast \ast }),(y^{\ast },y^{\ast \ast })\right\rfloor :=\langle
y^{\ast },x^{\ast \ast }\rangle +\langle x^{\ast },y^{\ast \ast }\rangle $. $%
(B,\lfloor \cdot ,\cdot \rfloor )$ is an SSD space. Clearly, for all $%
(y^{\ast },y^{\ast \ast })\in E^{\ast }\times E^{\ast \ast }$, $q(y^{\ast
},y^{\ast \ast })=\langle y^{\ast },y^{\ast \ast }\rangle $. Now $\iota (A)$
is $q$--positive, $A$ is of type (NI) exactly when $\Phi _{\iota (A)}\geq q$
on $B$, and $A$ is unique exactly when $\iota (A)$ is premaximally $q$%
--positive. In this case, we write $\iota (A)^{\pi }$ for the unique
maximally monotone subset of $E^{\ast }\times E^{\ast \ast }$ that contains $%
\iota (A)$.
\end{definition}

Corollary \ref{106}(a) appears in \cite{301}, and Corollary \ref{106}(c)
appears in \cite[Theorem 1.6]{MAS09}.

\begin{corollary}
\label{106}Let $E$ be a nonzero Banach space and $A$ be a nonempty monotone
subset of $E\times E^{\ast }$.

\noindent \textrm{(a) }If $A$ is of type (NI) then $A$ is unique and $\iota
(A)^{\pi }=\mathcal{P}_{q}\left( \Phi _{\iota (A)}\right) $.

\noindent \textrm{(b) }If $\iota (A)^{\pi }$ is an affine subset of $E^{\ast
}\times E^{\ast \ast }$ then $A$ is unique.

\noindent \textrm{(c)} Let $A$ be unique. Then either $A$ is of type (NI),
or 
\begin{equation}
\iota (A)^{\pi }=\{(y^{\ast },y^{\ast \ast })\in E^{\ast }\times E^{\ast
\ast }\colon \ \inf\nolimits_{(a,a^{\ast })\in A}\langle a^{\ast }-y^{\ast },%
\widehat{a}-y^{\ast \ast }\rangle >-\infty \}  \label{205}
\end{equation}%
and $\iota (A)^{\pi }$ is an affine subset of $E^{\ast }\times E^{\ast \ast
} $.

\noindent \textrm{(d)} Let $A$ be maximally monotone and unique. Then either 
$A$ is of type (NI), or $A$ is an affine subset of $E\times E^{\ast }$ and $%
A=\hbox{\rm dom}\,\varphi _{A}$, where $\varphi _{A}$ is the Fitzpatrick
function of $A$ in the usual sense.
\end{corollary}

\begin{proof}
(a), (b) and (c) are immediate from Lemmas \ref{102}\ and \ref{103} and
Theorem \ref{104}, and the terminology introduced in Definition \ref{105}.

(d). From (c) and the linearity of $\iota $, $\iota ^{-1}\left( \iota
(A)^{\pi }\right) $ is an affine subset of $E\times E^{\ast }$. Furthermore,
it is also easy to see that $\iota ^{-1}\left( \iota (A)^{\pi }\right) $ is
a monotone subset of $E\times E^{\ast }$. Since $A\subset \iota ^{-1}\left(
\iota (A)^{\pi }\right) $, the maximality of $A$ implies that $A=\iota
^{-1}\left( \iota (A)^{\pi }\right) $. Finally, it follows from (\ref{205})
that $\iota ^{-1}\left( i(A)^{\pi }\right) =\hbox{\rm dom}\,\varphi _{A}$.
\end{proof}

\subsection{Minimal convex functions bounded below by $q$}

This section extends some results of \cite{MS08}.

\begin{lemma}
\label{fund ineq}Let $B$ be an SSD space and $f:B\rightarrow \mathbb{R}\cup
\{+\infty \}$ be a proper convex function$.$ Then, for every $x,y\in B$ and
every $\alpha ,\beta \geq 0$ with $\alpha +\beta =1,$ one has%
\begin{equation*}
\alpha \max \left\{ f\left( x\right) ,q\left( x\right) \right\} +\beta \max
\left\{ f^{@}\left( y\right) ,q\left( y\right) \right\} \geq q\left( \alpha
x+\beta y\right) .
\end{equation*}
\end{lemma}

\begin{proof}
Using (\ref{FY}) one gets%
\begin{eqnarray*}
q\left( \alpha x+\beta y\right) &=&\alpha ^{2}q\left( x\right) +\alpha \beta
\lfloor x,y\rfloor +\beta ^{2}q\left( y\right) \\
&\leq &\alpha ^{2}q\left( x\right) +\alpha \beta \left( f(x)+f^{@}(y)\right)
+\beta ^{2}q\left( y\right) \\
&=&\alpha \left( \alpha q\left( x\right) +\beta f(x)\right) +\beta \left(
\alpha f^{@}(y)+\beta q\left( y\right) \right) \\
&\leq &\alpha \max \left\{ f\left( x\right) ,q\left( x\right) \right\}
+\beta \max \left\{ f^{@}\left( y\right) ,q\left( y\right) \right\} .
\end{eqnarray*}
\end{proof}

\begin{corollary}
\label{basic cor}Let $B$ be an SSD space, $f:B\rightarrow \mathbb{R\cup }%
\left\{ +\infty \right\} $ be a proper convex function such that $f\geq q$
and $x\in B.$ Then there exists a convex function $h:B\rightarrow \mathbb{%
R\cup }\left\{ +\infty \right\} $ such that%
\begin{equation*}
f\geq h\geq q\text{\qquad and\qquad }\max \left\{ f^{@}\left( x\right)
,q\left( x\right) \right\} \geq h\left( x\right) .
\end{equation*}
\end{corollary}

\begin{proof}
Let $h:=conv~\min \left\{ f,\delta _{\left\{ x\right\} }+\max \left\{
f^{@}\left( x\right) ,q\left( x\right) \right\} \right\} .$ Clearly, $h$ is
convex, $f\geq h,$ and $\max \left\{ f^{@}\left( x\right) ,q\left( x\right)
\right\} \geq h\left( x\right) ;$ so, we only have to prove that $h\geq q.$
Let $y\in B.$ Since the functions $f$ and $\delta _{\left\{ x\right\} }+\max
\left\{ f^{@}\left( x\right) ,q\left( x\right) \right\} $ are convex, we have%
\begin{eqnarray*}
h\left( y\right) &=&\inf_{\substack{ u,v\in B  \\ \text{ }\alpha ,\beta \geq
0,\text{ }\alpha +\beta =1  \\ \alpha u+\beta v=y}}\left\{ \alpha f\left(
u\right) +\beta \left( \delta _{\left\{ x\right\} }\left( v\right) +\max
\left\{ f^{@}\left( x\right) ,q\left( x\right) \right\} \right) \right\} \\
&=&\inf_{\substack{ u\in B  \\ \alpha ,\beta \geq 0,\text{ }\alpha +\beta =1 
\\ \alpha u+\beta x=y}}\left\{ \alpha f\left( u\right) +\beta \max \left\{
f^{@}\left( x\right) ,q\left( x\right) \right\} \right\} \\
&\geq &\inf_{\substack{ u\in B  \\ \alpha ,\beta \geq 0,\text{ }\alpha
+\beta =1  \\ \alpha u+\beta x=y}}q\left( \alpha u+\beta x\right) =q\left(
y\right) ,
\end{eqnarray*}%
the above inequality being a consequence of the assumption $f\geq q$ and
Lemma \ref{fund ineq}. We thus have $h\geq q.$
\end{proof}

\begin{theorem}
\label{minimal}Let $B$ be an SSD space and $f:B\rightarrow \mathbb{R\cup }%
\left\{ +\infty \right\} $ be a minimal element of the set of convex
functions minorized by $q.$ Then $f^{@}\geq f.$
\end{theorem}

\begin{proof}
It is easy to see that $f$ is proper. Let $x\in B$ and consider the function 
$h$ provided by Corollary \ref{basic cor}. By the minimality of $f,$ we
actually have $h=f;\ $on the other hand, from (\ref{FY}) it follows that $%
\frac{1}{2}\left( f(x)+f^{@}(x)\right) \geq \frac{1}{2}\lfloor x,x\rfloor
=q\left( x\right) .$ Therefore $f\left( x\right) =h\left( x\right) \leq \max
\left\{ f^{@}\left( x\right) ,q\left( x\right) \right\} \leq \max \left\{
f^{@}\left( x\right) ,\frac{1}{2}\left( f(x)+f^{@}(x)\right) \right\} ;$
from these inequalities one easily obtains that $f\left( x\right) \leq
f^{@}\left( x\right) .$
\end{proof}

\begin{proposition}
Let $B$ be an SSD space and $f:B\rightarrow \mathbb{R\cup }\left\{ +\infty
\right\} $ be a convex function such that $f\geq q$ and $f^{@}\geq q.$ Then%
\begin{equation*}
conv~\min \left\{ f,f^{@}\right\} \geq q.
\end{equation*}
\end{proposition}

\begin{proof}
Since $f$ and $f^{@}$ are convex, for every $x\in B$ we have%
\begin{eqnarray*}
conv~\min \left\{ f,f^{@}\right\} \left( x\right) &=&\inf_{\substack{ u,v\in
B  \\ \alpha ,\beta \geq 0,\text{ }\alpha +\beta =1  \\ \alpha u+\beta v=x}}%
\left\{ \alpha f\left( u\right) +\beta f^{@}\left( v\right) \right\} \\
&\geq &\inf_{\substack{ u,v\in B  \\ \alpha ,\beta \geq 0,\text{ }\alpha
+\beta =1  \\ \alpha u+\beta v=x}}q\left( \alpha u+\beta v\right) =q\left(
x\right) ,
\end{eqnarray*}%
the inequality following from the assumptions $f\geq q$ and $f^{@}\geq q$
and Lemma \ref{fund ineq}.
\end{proof}

\section{SSDB\ spaces}

We say that $\left( B,\lfloor \cdot ,\cdot \rfloor ,\Vert \cdot \Vert
\right) $ is a symmetrically self-dual Banach ({SSDB}) space if $\left(
B,\lfloor \cdot ,\cdot \rfloor \right) $ is an SSD space, $\left( B,\Vert
\cdot \Vert \right) $ is a Banach space, the dual $B^{\ast }$ is exactly $%
\{\lfloor \cdot ,b\rfloor :b\in B\}$ and the map $i:B\rightarrow B^{\ast }$
defined by $i(b)=\lfloor \cdot ,b\rfloor $ is a surjective isometry. In this
case, the quadratic form $q$ is continuous. By \cite[Proposition 3]{JE08} we
know that every SSDB space is reflexive as a Banach space. If $A$ is convex
in an SSDB\ space then $A^{w}=\overline{A}.$

Let $B$ be an SSDB space. In this case, for a proper convex function $%
f:B\rightarrow \mathbb{R}\cup \{+\infty \}$ it is easy to see that $%
f^{@}=f^{\ast }\circ i,$ where $f^{\ast }:B^{\ast }\rightarrow \mathbb{R}%
\cup \{+\infty \}$ is the Banach space conjugate of $f.$ Define $%
g_{0}:B\rightarrow \mathbb{R}$ by $g_{0}(b):=\frac{1}{2}\left\Vert
b\right\Vert ^{2}.$ Then for all $b^{\ast }\in B^{\ast }$, $g_{0}^{\ast
}\left( b^{\ast }\right) =\frac{1}{2}\left\Vert b^{\ast }\right\Vert ^{2}.$

\subsection{A characterization of maximally $q$-positive sets in SSDB spaces}

\begin{lemma}
\label{P-q}The set $\mathcal{P}_{q}(g_{0})$\bigskip $=\left\{ x\in
B:g_{0}(x)=q\left( x\right) \right\} $ is maximally $q$-positive and the set 
$\mathcal{P}_{-q}(g_{0})$\bigskip $=\left\{ x\in B:g_{0}(x)=-q\left(
x\right) \right\} $ is maximally $-q$-positive.
\end{lemma}

\begin{proof}
To prove that $\mathcal{P}_{q}(g_{0})$ is maximally $q$-positive, apply \cite%
[Thm. 4.3(b)]{Si07} (see also \cite[Thm. 2.7]{JE08}) after observing that $%
g_{0}^{@}=g_{0}^{\ast }\circ i=g_{0}.$ Since replacing $q$ by $-q$ changes $%
\mathcal{P}_{q}(g_{0})$ into $\mathcal{P}_{-q}(g_{0}),$ it follows that $%
\mathcal{P}_{-q}(g_{0})$ is maximally $-q$-positive too.
\end{proof}

\bigskip

From now on, to distinguish the function $\Phi _{A}$ of $A\subset B$
corresponding to $q$ from that corresponding to $-q,$ we will use the
notations $\Phi _{q,A}$ and $\Phi _{-q,A},$ respectively. Notice that $\Phi
_{-q,\mathcal{P}_{-q}(g_{0})}$ is finite-valued;\ indeed,

\begin{eqnarray*}
\Phi _{-q,\mathcal{P}_{-q}(g_{0})}\left( x\right) &=&\sup_{a\in \mathcal{P}%
_{-q}(g_{0})}\left\{ \mathcal{-}\lfloor x,a\rfloor +q\left( a\right) \right\}
\\
&=&\sup_{a\in \mathcal{P}_{-q}(g_{0})}\left\{ -\left\langle x,i\left(
a\right) \right\rangle -g_{0}(a)\right\} \\
&=&\sup_{a\in \mathcal{P}_{-q}(g_{0})}\left\{ -\left\langle x,i\left(
a\right) \right\rangle -g_{0}^{\ast }\left( i\left( a\right) \right)
\right\} \leq g_{0}(x).
\end{eqnarray*}

\begin{theorem}
Let $B$ be an SSDB space and $A$ be a $q$-positive subset of $B$, and
consider the following statements:

(1) $A$ is maximally $q$-positive.

(2) $A+C=B$ for every maximally $-q$-positive set $C\subseteq B$ such that $%
\Phi _{-q,C}$ is finite-valued.

(3) There exists a set $C\subseteq B$ such that $A+C=B,$ and there exists $%
p\in C$ such that%
\begin{equation*}
q\left( z-p\right) <0\text{\qquad }\forall ~z\in C\setminus \left\{
p\right\} .
\end{equation*}

Then (1), (2) and (3) are equivalent.
\end{theorem}

\begin{proof}
\emph{(1) }$\Longrightarrow $\emph{\ (2).}\textbf{\ }Let $x_{0}\in B$ and $%
A^{\prime }:=A-\left\{ x_{0}\right\} .$ We have%
\begin{equation*}
\Phi _{q,A^{\prime }}\left( x\right) +\Phi _{-q,C}\left( -x\right) \geq
q\left( x\right) -q\left( -x\right) =0\text{\qquad }\forall ~x\in C.
\end{equation*}%
Hence, as $\Phi _{-q,C}$ is continuous because it is lower semicontinuous
and finite-valued, by the Fenchel-Rockafellar duality theorem there exists $%
y^{\ast }\in B^{\ast }$ such that%
\begin{equation*}
\Phi _{q,A^{\prime }}^{\ast }\left( y^{\ast }\right) +\Phi _{-q,C}^{\ast
}\left( y^{\ast }\right) \leq 0.
\end{equation*}%
Since, by Proposition \ref{prop0}\textit{(1)}, $\Phi _{q,A^{\prime }}^{\ast
}\circ i=\Phi _{q,A^{\prime }}^{@}\geq $ $\Phi _{q,A^{\prime }}$ and,
correspondingly, $\Phi _{-q,C}^{\ast }\circ \left( -i\right) =\Phi
_{-q,C}^{@}\geq \Phi _{-q,C},$ we thus have%
\begin{eqnarray*}
0 &\geq &\left( \Phi _{q,A^{\prime }}^{\ast }\circ i\right) \left(
i^{-1}\left( y^{\ast }\right) \right) +\left( \Phi _{-q,C}^{\ast }\circ
\left( -i\right) \right) \left( -i^{-1}\left( y^{\ast }\right) \right) \\
&\geq &\Phi _{q,A^{\prime }}\left( i^{-1}\left( y^{\ast }\right) \right)
+\Phi _{-q,C}\left( -i^{-1}\left( y^{\ast }\right) \right) \geq q\left(
i^{-1}\left( y^{\ast }\right) \right) -q\left( -i^{-1}\left( y^{\ast
}\right) \right) =0.
\end{eqnarray*}%
Therefore%
\begin{equation*}
\Phi _{q,A^{\prime }}\left( i^{-1}\left( y^{\ast }\right) \right) =q\left(
i^{-1}\left( y^{\ast }\right) \right) \text{ and }\Phi _{-q,C}\left(
-i^{-1}\left( y^{\ast }\right) \right) =-q\left( -i^{-1}\left( y^{\ast
}\right) \right) ,
\end{equation*}%
that is,%
\begin{equation*}
i^{-1}\left( y^{\ast }\right) \in A^{\prime }\text{ and }-i^{-1}\left(
y^{\ast }\right) \in C,
\end{equation*}%
which implies that%
\begin{equation*}
x_{0}=x_{0}+i^{-1}\left( y^{\ast }\right) -i^{-1}\left( y^{\ast }\right) \in
x_{0}+A^{\prime }+C=A+C.
\end{equation*}

\emph{(2) }$\Longrightarrow $\emph{\ (3). }Take $C:=\mathcal{P}_{-q}(g_{0})$
(see Lemma \ref{P-q}) and $p:=0.$

\emph{(3) }$\Longrightarrow $\emph{\ (1). }Let $x\in A^{\pi },$ and take $p$
as in \textit{(3)}. Since $x+p\in B=A+C,$ we have $x+p=y+z$ for some $y\in A$
and $z\in C.$ We have $x-y=z-p;$ hence, since $x\in A^{\pi }$ and $y\in A,$
we get $0\leq q\left( x-y\right) =q\left( z-p\right) \leq 0.$ Therefore $%
q\left( z-p\right) =0,$ which implies $z=p.$ Thus from $x+p=y+z$ we obtain $%
x=y\in A.$ This proves that $A^{\pi }\subset A,$ which, together with the
fact that $A$ is $q$-positive, shows that $A$ is maximally $q$-positive.
\end{proof}

\begin{corollary}
One has%
\begin{equation*}
\mathcal{P}_{q}(g_{0})+\mathcal{P}_{-q}(g_{0})=B.
\end{equation*}
\end{corollary}

\begin{proof}
Since the set $\mathcal{P}_{q}(g_{0})$ is maximally $q$-positive by Lemma %
\ref{P-q}, the result follows from the implication \emph{(1) }$%
\Longrightarrow $\emph{\ (2) }in the preceding theorem.
\end{proof}

\subsection{Minimal convex functions on SSDB spaces bounded below by $q$}

\begin{theorem}
If $B$ is an SSDB\ space and $f:B\rightarrow \mathbb{R\cup }\left\{ +\infty
\right\} $ is a minimal element of the set of convex functions minorized by $%
q$ then $f=\Phi _{M}$ for some maximally $q$-positive set $M\subset B.$
\end{theorem}

\begin{proof}
We first observe that $f$ is lower semicontinuous; indeed, this is a
consequence of its minimality and the fact that its lower semicontinuous
closure is convex and minorized by $q$ because $q$ is continuous. By Theorem %
\ref{minimal} and \cite[Thm. 4.3(b)]{Si07} (see also \cite[Thm. 2.7]{JE08}),
the set $\mathcal{P}_{q}(f)$ is maximally $q$-positive, and hence $\Phi _{%
\mathcal{P}_{q}(f)}\geq q.$ From \cite[Thm. 2.2]{JE08} we deduce that $\Phi
_{\mathcal{P}_{q}(f)}\leq f,$ which, by the minimality of $f,$ implies that $%
\Phi _{\mathcal{P}_{q}(f)}=f.$
\end{proof}

\section{Examples}

\subsection{Lipschitz mappings between Hilbert spaces}

Let $K>0.$ Let $H_{1},H_{2}$ be two real Hilbert spaces and let $f:D\subset
H_{1}\rightarrow H_{2}$ be a $K$-Lipschitz mapping, i.e. 
\begin{equation}
\Vert f(x_{1})-f(y_{1})\Vert _{H_{2}}\leq K\Vert x_{1}-y_{1}\Vert
_{H_{1}},\qquad \forall \;x_{1},y_{1}\in D.  \label{one}
\end{equation}

\begin{remark}
\label{rem0}It is well known that there exists an extension $\tilde{f}%
:H_{1}\rightarrow H_{2}$ which is $K$-Lipschitz (see \cite{K34, V45}). Let $%
D\subset H_{1}$; we will denote by $\mathcal{F(}D\mathcal{)}$ the family of $%
K$-Lipschitz mappings defined on $D$ and by $\mathcal{F}:\mathcal{=F}\left(
H_{1}\right) $ the family of $K$-Lipschitz mappings defined everywhere on $%
H_{1}$.
\end{remark}

\begin{proposition}
\label{pzero}Let $H_{1},H_{2}$ be two real Hilbert spaces, let $%
B=H_{1}\times H_{2}$ and let $\left\lfloor \cdot ,\cdot \right\rfloor
:B\times B\rightarrow \mathbb{R}$ be the bilinear form defined by 
\begin{equation}
\left\lfloor (x_{1},x_{2}),(y_{1},y_{2})\right\rfloor =K^{2}\langle
x_{1},y_{1}\rangle _{H_{1}}-\langle x_{2},y_{2}\rangle _{H_{2}}.
\label{zero}
\end{equation}%
Then

(1) A nonempty set $A\subset B$ is $q$-positive if and only if there exists $%
f\in \mathcal{F}(P_{H_{1}}(A))$ such that $A=graph(f)$;

(2) A set $A\subset B$ is maximally $q$-positive if and only if there exists 
$f\in \mathcal{F}$ such that $A=graph(f)$.
\end{proposition}

\begin{proof}
\textit{(1). }If $A=graph(f)$ with $f\in \mathcal{F}(P_{H_{1}}(A)),$ it is
easy to see that $A$ is $q$-positive.

Assume that $A\subset B$ is $q$-positive. From the definition we have that
for all $(x_{1},y_{1}),(x_{2},y_{2})\in A$, 
\begin{equation*}
0\leq q\left( (x_{1},y_{1})-(x_{2},y_{2})\right) =\frac{1}{2}\left(
K^{2}\Vert x_{1}-x_{2}\Vert _{H_{1}}^{2}-\Vert y_{1}-y_{2}\Vert
_{H_{2}}^{2}\right) .
\end{equation*}%
Equivalently,%
\begin{equation}
\Vert y_{1}-y_{2}\Vert _{H_{2}}\leq K\Vert x_{1}-x_{2}\Vert _{H_{1}}.
\label{two}
\end{equation}%
For $x\in P_{H_{1}}(A)$ we define $f(x)=\{y:(x,y)\in A\}$. We will show that 
$f$ is a $K$-Lipschitz mapping. If $y_{1},y_{2}\in f(x),$ from (\ref{two}) $%
y_{1}=y_{2}$, so $f$ is single-valued. Now, for $x_{1},x_{2}\in P_{H_{1}}(A)$
from (\ref{two}) we have that 
\begin{equation*}
\Vert f(x_{1})-f(x_{2})\Vert _{H_{2}}\leq K\Vert x_{1}-x_{2}\Vert _{H_{1}},
\end{equation*}%
which shows that $f\in \mathcal{F}(P_{H_{1}}(A))$.\newline
\textit{(2). }Let $A\subset B$ be maximally $q$-positive. From \textit{(1)},
there exists $f\in \mathcal{F}(P_{H_{1}}(A))$ such that $A=graph(f)$, and
from the Kirszbraun-Valentine extension theorem \cite{K34, V45} $f$ has a $K$%
-Lipschitz extension $\tilde{f}$ defined everywhere on $H_{1}$; since $graph(%
\tilde{f})$ is also $q$-positive we must have $f=\tilde{f}$. Now, let $f\in 
\mathcal{F}$ and $(x,y)\in H_{1}\times H_{2}$ be $q$-positively related to
every point in $graph(f)$. We have that $graph(f)\cup \{(x,y)\}$ is $q$%
-positive, so from (1) we easily deduce that $y=f(x)$. This finishes the
proof of \textit{(2).}
\end{proof}

\bigskip

Clearly, the $w\left( B,B\right) $ topology of the SSD space $(B,\lfloor
\cdot ,\cdot \rfloor )$ coincides with the weak topology of the product
Hilbert space $H_{1}\times H_{2}.$ Therefore, every $q$-representable set is
closed, so that it corresponds to a $K$-Lipschitz mapping with closed graph.
Notice that, by the Kirszbraun-Valentine extension theorem, a $K$-Lipschitz
mapping between two Hilbert spaces has a closed graph if and only if its
domain is closed. The following example shows that not every $K$-Lipschitz
mapping with closed domain has a $q$-representable graph.

\begin{example}
Let $H_{1}:=\mathbb{R}=:H_{2}$ and let $f:\left\{ 0,1\right\} \rightarrow
H_{2}$ be the restriction of the identity mapping$.$ Clearly, $f$ is
nonexpansive, so we will consider the SSD space corresponding to $K=1.$ Then
we will show that the smallest $q$-representable set containing $graph(f)$
is the graph of the restriction $\widehat{f}$ of the identity to the closed
interval $\left[ 0,1\right] .$ Notice that this graph is indeed $q$%
-representable, since the lsc function $\delta _{graph(\widehat{f})}$
belongs to $\mathcal{PC}_{q}\left( B\right) $ and one has $graph(\widehat{f}%
)=\mathcal{P}_{q}\left( \delta _{graph(\widehat{f})}\right) .$ We will see
that $graph(\widehat{f})\subset \mathcal{P}_{q}\left( \varphi \right) $ for
every $\varphi \in \mathcal{PC}_{q}\left( B\right) $ such that $%
graph(f)\subset \mathcal{P}_{q}\left( \varphi \right) .$ Indeed, for $t\in %
\left[ 0,1\right] $ one has $\varphi \left( t,t\right) \leq \left(
1-t\right) \varphi \left( 0,0\right) +t\varphi \left( 1,1\right) =\left(
1-t\right) q\left( 0,0\right) +tq\left( 1,1\right) =0=q\left( t,t\right) ;\ $%
hence $\left( t,t\right) \in \mathcal{P}_{q}\left( \varphi \right) ,$ which
proves the announced inclusion.
\end{example}

Our next two results provide sufficient conditions for $q$-representability
in the SSD space we are considering.

\begin{proposition}
\label{closed domain}Let $H_{1},H_{2},$ $B$ and $\left\lfloor \cdot ,\cdot
\right\rfloor $ be as in Proposition \ref{pzero} and let $f:D\subset
H_{1}\rightarrow H_{2}$ be a $K^{\prime }$-Lipschitz mapping, with $%
0<K^{\prime }<K.\ $If $D$ is nonempty and closed, then $graph(f)$ is $q$%
-representable.
\end{proposition}

\begin{proof}
We will prove that $graph(f)$ coincides with the intersection of all the
graphs of $K$-Lipschitz extensions $\tilde{f}$ of $f$ to the whole of $%
H_{1}. $ Since any such graph is maximally $q$-positive, we have $graph(%
\tilde{f})=\mathcal{P}_{q}\left( \Phi _{graph(\tilde{f})}\right) ;$ hence
that intersection is equal to $\mathcal{P}_{q}\left( \varphi \right) ,$
where $\varphi $ denotes the supremum of all the functions $\Phi _{graph(%
\tilde{f})};$ so the considered intersection is $q$-representable. As one
clearly has $graph(f)\subset \mathcal{P}_{q}\left( \varphi \right) ,$ we
will only prove the opposite inclusion. Let $\left( x_{1},x_{2}\right) \in 
\mathcal{P}_{q}\left( \varphi \right) .$ Then $\tilde{f}\left( x_{1}\right)
=x_{2}$ for every $\tilde{f},$ so it will suffice to prove that $x_{1}\in D.$
Assume, towards a contradiction, that $x_{1}\notin D.$ By the
Kirszbraun-Valentine extension theorem, some $\tilde{f}$ is actually $%
K^{\prime }$-Lipschitz. Take any $y\in H_{2}\setminus \left\{ x_{2}\right\} $
in the closed ball with center $x_{2}$ and radius $\left( K-K^{\prime
}\right) \inf_{x\in D}\left\Vert x-x_{1}\right\Vert _{H_{1}}.$ This number
is indeed strictly positive, since $D$ is closed. Let $f_{y}$ be the
extension of $f$ to $D\cup \left\{ x_{1}\right\} $ defined by $f_{y}\left(
x_{1}\right) =y.$ This mapping is $K$-Lipschitz, since for every $x\in D$
one has $\left\Vert f_{y}\left( x\right) -f_{y}\left( x_{1}\right)
\right\Vert _{H_{2}}=\left\Vert f\left( x\right) -y\right\Vert _{H_{2}}\leq
\left\Vert f\left( x\right) -x_{2}\right\Vert _{H_{2}}+\left\Vert
x_{2}-y\right\Vert _{H_{2}}=\left\Vert \tilde{f}\left( x\right) -\tilde{f}%
\left( x_{1}\right) \right\Vert _{H_{2}}+\left( K-K^{\prime }\right)
\left\Vert x-x_{1}\right\Vert _{H_{1}}\leq K^{\prime }\left\Vert
x-x_{1}\right\Vert _{H_{1}}+\left( K-K^{\prime }\right) \left\Vert
x-x_{1}\right\Vert _{H_{1}}=K\left\Vert x-x_{1}\right\Vert _{H_{1}}.$ Using
again the Kirszbraun-Valentine extension theorem, we get the existence of a $%
K$-Lipschitz extension $\widetilde{f_{y}}\in \mathcal{F}$ of $f_{y}.$ Since $%
\left( x_{1},x_{2}\right) \in \mathcal{P}_{q}\left( \varphi \right) \subset
graph(\widetilde{f_{y}}),$ we thus contradict $\widetilde{f_{y}}\left(
x_{1}\right) =f_{y}\left( x_{1}\right) =y.$
\end{proof}

\begin{proposition}
Let $H_{1},H_{2},$ $B$ and $\left\lfloor \cdot ,\cdot \right\rfloor $ be as
in Proposition \ref{pzero} and let $f:D\subset H_{1}\rightarrow H_{2}$ be a $%
K$-Lipschitz mapping. If $D$ is nonempty, convex, closed and bounded, then $%
graph(f)$ is $q$-representable.
\end{proposition}

\begin{proof}
As in the proof of Proposition \ref{closed domain}, it will suffice to show
that $graph(f)$ coincides with the intersection of all the graphs of $K$%
-Lipschitz extensions $\tilde{f}$ of $f$ to the whole of $H_{1},$ and we
will do it by proving that for every point $\left( x_{1},x_{2}\right) $ in
this intersection one necessarily has $x_{1}\in D$. If we had $x_{1}\notin
D, $ by the Hilbert projection theorem there would be a closest point $%
\overline{x}$ to $x_{1}$ in $D,$ characterized by the condition $%
\left\langle x-\overline{x},x_{1}-\overline{x}\right\rangle \leq 0$ for all $%
x\in D.$ Let $C:=\sup_{x\in D}\left\{ \left\Vert x-x_{1}\right\Vert
+\left\Vert x-\overline{x}\right\Vert \right\} .$ Since $x_{1}\neq \overline{%
x}$ and $D$ is nonempty and bounded, $C\in \left( 0,+\infty \right) .$ For
every $x\in D$ we have $\left\Vert x-x_{1}\right\Vert -\left\Vert x-%
\overline{x}\right\Vert =\frac{\left\Vert x-x_{1}\right\Vert ^{2}-\left\Vert
x-\overline{x}\right\Vert ^{2}}{\left\Vert x-x_{1}\right\Vert +\left\Vert x-%
\overline{x}\right\Vert }=\frac{\left\Vert x_{1}-\overline{x}\right\Vert
^{2}+2\left\langle x-\overline{x},\overline{x}-x_{1}\right\rangle }{%
\left\Vert x-x_{1}\right\Vert +\left\Vert x-\overline{x}\right\Vert }\geq 
\frac{\left\Vert x_{1}-\overline{x}\right\Vert ^{2}}{\left\Vert
x-x_{1}\right\Vert +\left\Vert x-\overline{x}\right\Vert }\geq \frac{%
\left\Vert x_{1}-\overline{x}\right\Vert ^{2}}{C}.$ Take $y\in
H_{2}\setminus \left\{ x_{2}\right\} $ in the closed ball with center $%
f\left( \overline{x}\right) $ and radius $\frac{K\left\Vert x_{1}-\overline{x%
}\right\Vert ^{2}}{C}.$ Let $f_{y}$ be the extension of $f$ to $D\cup
\left\{ x_{1}\right\} $ defined by $f_{y}\left( x_{1}\right) =y.$ This
mapping is $K$-Lipschitz, since for every $x\in D$ one has $\left\Vert
f_{y}\left( x\right) -f_{y}\left( x_{1}\right) \right\Vert
_{H_{2}}=\left\Vert f\left( x\right) -y\right\Vert _{H_{2}}\leq \left\Vert
f\left( x\right) -f\left( \overline{x}\right) \right\Vert
_{H_{2}}+\left\Vert f\left( \overline{x}\right) -y\right\Vert _{H_{2}}\leq
K\left\Vert x-\overline{x}\right\Vert _{H_{1}}+\left\Vert f\left( \overline{x%
}\right) -y\right\Vert _{H_{2}}\leq K\left\Vert x-\overline{x}\right\Vert
_{H_{1}}+\frac{K\left\Vert x_{1}-\overline{x}\right\Vert ^{2}}{C}\leq
K\left\Vert x-\overline{x}\right\Vert _{H_{1}}+K\left( \left\Vert
x-x_{1}\right\Vert -\left\Vert x-\overline{x}\right\Vert \right)
=K\left\Vert x-x_{1}\right\Vert .$ The proof finishes by applying the same
reasoning as at the end of the proof of Proposition \ref{closed domain}.
\end{proof}

\bigskip

In this framework, for $A:=graph(f)$ the function $\Phi _{A}$ is given by 
\begin{equation*}
\Phi _{A}(x_{1},x_{2})=\frac{1}{2}\sup_{a_{1}\in domf}\{-K^{2}\Vert
a_{1}-x_{1}\Vert _{H_{1}}^{2}+\Vert f(a_{1})-x_{2}\Vert _{H_{2}}^{2}\}+\frac{%
K^{2}}{2}\Vert x_{1}\Vert ^{2}-\frac{1}{2}\Vert x_{2}\Vert ^{2}.
\end{equation*}

It is also evident that $\left( B,\left\lfloor \cdot ,\cdot \right\rfloor
,\left\Vert \cdot \right\Vert \right) $ is an SSDB space if and only if $%
K=1. $

\subsection{Closed sets in a Hilbert space}

Let $\left( H,\left\langle \cdot ,\cdot \right\rangle \right) $ be a Hilbert
space.and denote by $\left\Vert \cdot \right\Vert $ the induced norm on $H.$
Clearly, $\left( H,\left\langle \cdot ,\cdot \right\rangle ,\left\Vert \cdot
\right\Vert \right) $ is an SSDB\ space, and the associated quadratic form $%
q $ is given by $q\left( x\right) =\frac{1}{2}\left\Vert x\right\Vert ^{2}.$
Since $q$ is nonnegative, every nonempty set $A\subset H$ is $q$-positive.
We further have:

\begin{proposition}
A nonempty set $A\subset H$ is $q$-representable if and only if it is closed.
\end{proposition}

\begin{proof}
The "only if" statement being obvious, we will only prove the converse.
Define $h:H\rightarrow \mathbb{R\cup }\{+\infty \}$ by%
\begin{equation*}
h\left( x\right) =\sup_{y\in H}\left\{ q\left( y\right) +\left\langle
y,x-y\right\rangle +\frac{1}{2}d_{A}^{2}\left( y\right) \right\} ,
\end{equation*}%
with $d_{A}\left( y\right) :=\inf_{a\in A}\left\Vert y-a\right\Vert .$
Clearly, $h$ is convex and lsc$.$

For all $x\in H,$%
\begin{equation*}
h\left( x\right) \geq q\left( x\right) +\left\langle x,x-x\right\rangle +%
\frac{1}{2}d_{A}^{2}\left( x\right) =q\left( x\right) +\frac{1}{2}%
d_{A}^{2}\left( x\right) \geq q\left( x\right) ,
\end{equation*}%
which implies that $h\geq q$ and $\mathcal{P}_{q}(h)\subset A.$ We will
prove that $h$ represents $A,$ that is,%
\begin{equation}
A=\mathcal{P}_{q}(h).  \label{closed representable}
\end{equation}%
To prove the inclusion $\subset $ in (\ref{closed representable}), let $x\in
A.$ Then, for all $y\in H,$%
\begin{eqnarray*}
q\left( y\right) +\left\langle y,x-y\right\rangle +\frac{1}{2}%
d_{A}^{2}\left( y\right) &\leq &\frac{1}{2}\left\Vert y\right\Vert
^{2}+\left\langle y,x-y\right\rangle +\frac{1}{2}\left\Vert y-x\right\Vert
^{2}=\frac{1}{2}\left\Vert x\right\Vert ^{2} \\
&=&q\left( x\right) ,
\end{eqnarray*}%
which proves that $h\left( x\right) \leq q\left( x\right) .$ Hence, as $%
h\geq q,$ the inclusion $\subset $ holds in (\ref{closed representable}). We
have thus proved (\ref{closed representable}), which shows that $A$ is $q$%
-representable.
\end{proof}

\begin{proposition}
\label{prop7} Let $\emptyset \neq A\subset H$. Then

(1) $\Phi _{A}(x)=\frac{1}{2}\Vert x\Vert ^{2}-\frac{1}{2}d_{A}^{2}(x)$;

(2) $\Phi _{A}^{@}(x)=\frac{1}{2}\Vert x\Vert ^{2}+\frac{1}{2}\sup_{b\in
H}\{d_{A}^{2}(b)-\Vert x-b\Vert ^{2}\}$;

(3) $\Phi _{A}^{@}(x)=\frac{1}{2}\Vert x\Vert ^{2}\Leftrightarrow x\in 
\overline{A}$;

(4) $G_{\Phi _{A}}=\left\{ x\in H:\sup_{b\in H}\{d_{A}^{2}(b)-\Vert b-x\Vert
^{2}\}=d_{A}^{2}(x)\right\} $
\end{proposition}

\begin{theorem}
\label{weir} Let $\emptyset \neq A\subset H$ be such that $A=G_{\Phi _{A}}$,
and let $a_{1}$, $a_{2}\in A$ be two different points, $x=\frac{1}{2}%
(a_{1}+a_{2})$ and $r=\frac{1}{2}\Vert a_{1}-a_{2}\Vert $. Denote by $%
B_{r}(x)$ the open ball with center $x$ and radius $r.$ Then, 
\begin{equation*}
B_{r}(x)\cap A\neq \emptyset .
\end{equation*}
\end{theorem}

\begin{proof}
Suppose that 
\begin{equation}
A\cap B_{r}(x)=\emptyset ,  \label{pr1}
\end{equation}%
so, we must have $d_{A}^{2}(x)=\Vert x-a_{1}\Vert ^{2}=\Vert x-a_{2}\Vert
^{2}$.\newline
For $b\in H$, we have%
\begin{equation*}
\text{either }\langle b-x,x-a_{1}\rangle \leq 0\;\text{or}\;\langle
b-x,x-a_{2}\rangle \leq 0.
\end{equation*}%
If $\langle b-x,x-a_{1}\rangle \leq 0$, 
\begin{equation*}
d_{A}^{2}(b)-\Vert b-x\Vert ^{2}\leq \Vert b-a_{1}\Vert ^{2}-\Vert b-x\Vert
^{2}\leq \Vert x-a_{1}\Vert ^{2}=d_{A}^{2}(x).
\end{equation*}%
If $\langle b-x,x-a_{2}\rangle \leq 0$, 
\begin{equation*}
d_{A}^{2}(b)-\Vert b-x\Vert ^{2}\leq \Vert b-a_{2}\Vert ^{2}-\Vert b-x\Vert
^{2}\leq \Vert x-a_{2}\Vert ^{2}=d_{A}^{2}(x).
\end{equation*}%
Thus, we deduce that 
\begin{equation*}
\sup_{b\in H}\{d_{A}^{2}(b)-\Vert b-x\Vert ^{2}\}=d_{A}^{2}(x),
\end{equation*}%
hence by Proposition \ref{prop7}\textit{(4)} $x\in G_{\Phi _{A}}=A$, which
is a contradiction with (\ref{pr1}).
\end{proof}

\begin{corollary}
Let $H=\mathbb{R}$ and $\emptyset \neq A\subset \mathbb{R}$. Then, 
\begin{equation*}
A=G_{\Phi _{A}}\text{ if and only if }A\ \text{is closed and convex.}
\end{equation*}
\end{corollary}

\begin{proof}
($\Longrightarrow $) Since $A=G_{\Phi _{A}}$, $A$ is closed. Assume that $A$
is not convex, so there exists $a_{1},a_{2}\in A$ such that $]a_{1},a_{2}[$ $%
\cap A=\emptyset $, hence 
\begin{equation*}
A\cap B_{r}(x)=\emptyset ,\hbox{ with }x=\frac{1}{2}(a_{1}+a_{2})\hbox{ and }%
r=\frac{1}{2}|a_{1}-a_{2}|,
\end{equation*}%
which contradicts Theorem \ref{weir}. Thus $A$ is convex.

($\Longleftarrow $) Since $A$ is closed, it is $q$-positive;\ hence we can
apply Theorem \ref{theorem1}\textit{(2)}.
\end{proof}

\bigskip

We will show with a simple example that, leaving aside the case $B=\mathbb{R}
$, in general $A=G_{\Phi _{A}}$ does not imply that $A$ is convex.

\begin{example}
Let $H=\mathbb{R}^{2},$ and let $A=\left\{ \left( x_{1},x_{2}\right) \in 
\mathbb{R}^{2}:x_{1}x_{2}=0\right\} .$ We will show that $A=G_{\Phi _{A}}.$
Let $x=\left( x_{1},x_{2}\right) \in \mathbb{R}^{2}\setminus A.$ Then%
\begin{equation*}
d_{A}\left( x\right) =\min \left\{ \left\vert x_{1}\right\vert ,\left\vert
x_{2}\right\vert \right\} .
\end{equation*}%
If $\lambda \in \mathbb{R}$, let $f\left( \lambda \right) :=d_{A}^{2}\left(
\lambda x\right) -\left\Vert \lambda x-x\right\Vert ^{2}=\lambda
^{2}d_{A}^{2}\left( x\right) -\left( \lambda -1\right) ^{2}\left\Vert
x\right\Vert ^{2}.$ Then $f^{\prime }\left( 1\right) =2d_{A}^{2}\left(
x\right) >0$ and so, if $\lambda $ is slightly greater than $1,$ $f\left(
\lambda \right) >f\left( 1\right) ,$ that is to say, $d_{A}^{2}\left(
\lambda x\right) -\left\Vert \lambda x-x\right\Vert ^{2}>d_{A}^{2}\left(
x\right) .$ Hence we have%
\begin{equation*}
\sup_{y\in H}\left\{ d_{A}^{2}\left( y\right) -\left\Vert y-x\right\Vert
^{2}\right\} >d_{A}^{2}\left( x\right) ;
\end{equation*}%
thus, by Proposition \ref{prop7}\textit{(4)}, $x\notin G_{\Phi _{A}}.$ We
deduce that $A=G_{\Phi _{A}}$, and clearly $A$ is not convex.
\end{example}

\end{document}